\documentclass{article}

\usepackage{arxiv}

\usepackage[utf8]{inputenc} 
\usepackage[T1]{fontenc}    
\usepackage{url}            
\usepackage{booktabs}       
\usepackage{amsfonts}       
\usepackage{nicefrac}       
\usepackage{microtype}      
\usepackage{lipsum}		
\usepackage{doi}

\usepackage{amssymb,amsmath,amsthm,amsbsy}
\usepackage{systeme}
\usepackage{physics}
\usepackage{hyperref} 
\usepackage{graphicx}
\usepackage{epstopdf}
\usepackage{epsfig}
\usepackage{multicol}
\usepackage{subfigure}
\usepackage{multirow}
\usepackage{algorithm,algorithmic}
\usepackage{bm}
\usepackage{setspace}
\usepackage{color}
\usepackage{lineno}
\usepackage{enumerate}
\newtheorem{definition}{Definition}[section]
\newtheorem{proposition}{Proposition}[section]
\usepackage{tcolorbox}

\title{Energy-based error bound of physics-informed \\ neural network solutions in elasticity}

\author{  \href{https://orcid.org/0000-0002-5541-437X}{\includegraphics[scale=0.06]{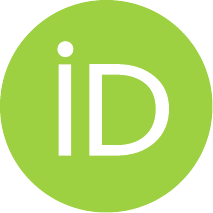}Mengwu Guo}\thanks{Corresponding author.}\\
	Department of Applied Mathematics \\
	University of Twente \\
	\texttt{m.guo@utwente.nl} \\
	\And 
	\href{https://orcid.org/0000-0003-2659-0507}{\includegraphics[scale=0.06]{orcid.pdf}Ehsan Haghighat}\thanks{The two authors contributed equally to this work.} \\
	Department of Civil and Environmental Engineering \\
	Massachusetts Institute of Technology \\
	\texttt{ehsanh@mit.edu} \\
}

\date{}

\renewcommand{\shorttitle}{Error bound of PINN solutions in elasticity}

\hypersetup{
pdftitle={Error bound of PINN solutions in elasticity},
pdfauthor={Mengwu Guo},
}

\begin{document}
\maketitle

\vspace{-10mm}
\begin{tcolorbox}
	A later version has been published in the \emph{Journal of Engineering Mechanics}. Cite the paper as:
	
	M. Guo and E. Haghighat. Energy-based error bound of physics-informed neural network solutions in elasticity. \emph{Journal of Engineering Mechanics}, 148(8):04022038, 2022. \textbf{DOI}: 10.1061/(ASCE)EM.1943-7889.0002121
	
	\smallskip
	This material may be downloaded for personal use only. Any other use requires prior permission of the American Society of Civil Engineers. This material may be found at \url{https://doi.org/10.1061/(ASCE)EM.1943-7889.0002121}.

\end{tcolorbox}

\vspace{3mm}

 \begin{abstract}
   An energy-based \emph{a posteriori} error bound is proposed for the physics-informed neural network solutions of elasticity problems. An admissible displacement-stress solution pair is obtained from a mixed form of physics-informed neural networks, and the proposed error bound is formulated as the constitutive relation error defined by the solution pair. Such an error estimator provides an upper bound of the global error of neural network discretization. The bounding property, as well as the asymptotic behavior of the physics-informed neural network solutions, are studied in a demonstrating example.
 \end{abstract}

\keywords{physics-informed neural network  \and constitutive relation error  \and \emph{a posteriori} error estimation \and machine learning}

\section{Introduction}

During the past decade, the application of deep neural networks, also known as deep learning, has gained a significant momentum for a variety of tasks including image classification \cite{deng2009imagenet,krizhevsky2012imagenet}, speech recognition \cite{graves2013speech,amodei2016speech}, autonomous driving \cite{sallab2017driving,grigorescu2020driving}, and e-commerce \cite{ha2016ecommerce,shankar2017ecommerce}, to name a few, see \cite{lecun2015deep,goodfellow2016deep} for more details. In a recent study, it was shown that neural networks can also be used for the solution and identification of partial differential equations \cite{raissi2019pinn,raissi2018hidden}. To that end, the solution space is constructed through deep neural network approximations and the partial derivatives are evaluated using automatic differentiation \cite{baydin2017automatic}. A loss function, in the form of a mean-squared error norm (MSE), is constructed to include the differential equations as well as the initial and boundary conditions. Minimization of such a loss function on a sampling grid results in an approximate solution to the problem under study. This approach is now commonly known as \emph{physics-informed neural networks} (PINNs), and the preliminary studies in \cite{raissi2019pinn,raissi2020pinn} have driven a great attention to this approach.

More recently, the methodology of PINNs has been widely applied in the contexts of both forward and inverse problems of fluid mechanics \cite{raissi2020pinn,yang2019fluid,jin2020fluid,mao2020fluid}, solid mechanics \cite{haghighat2020solid,rao2020solid}, heat transfer \cite{zhao2020prediction,haghighat2020sciann}, flow in porous media \cite{kadeethum2020biot,bekele2020biot} and so on. Moreover, the PINNs have also been investigated in their variational or fractional forms \cite{kharazmi2019variational,kharazmi2020varitionalhp,pang2019fractionals,yang2018generative}. In most of these studies, the PINN methodology is merely used as a tool and adopted to different applications. There is, however, a clear need for fundamental analysis of this methodology and its performance. With the aid of an energy-based \emph{a posteriori} error bound, we aim to preliminarily investigate the generalization errors and asymptotic behaviors of the PINN solutions of elasticity problems in this work.

To quantify the discretization error in an approximate solution, \emph{a posteriori} error estimation has been intensively studied for finite element methods \cite{ainsworth2011posteriori}. Several families of \emph{a posteriori} error estimators have been formulated, such as the explicit residual-based error estimator \cite{babuvska1978posteriori}, the implicit residual-based error estimator \cite{bank1985some}, the recovery-based error estimator \cite{zienkiewicz1987simple}, the hierarchical estimator \cite{deuflhard1989concepts}, and the constitutive relation error (CRE) estimator \cite{ladeveze1983error}. Among these existing error estimators, the CRE is claimed to provide guaranteed, rigorous, energy-based bounds of the discretization errors in finite element solutions \cite{ladeveze2005mastering,ladeveze2016constitutive}. In this work, we consider a mixed form of PINN approximation for elasticity problems, where the CRE estimation can be naturally employed to formulate an energy-based upper bound of the discretization errors by neural networks. Asymptotic behaviors of the PINN approximation can thus be observed and assessed through the energy-based error bounds given by the CRE.

Following the introduction, a model problem of elasticity and its mix form of PINN solution is introduced in Section 2. An energy-based error bound is formulated for the model problem based on the CRE estimation in Section 3. In Section 4, the proposed error bound is demonstrated by a 2D elasticity problem, and its asymptotic behaviors are discussed. Finally, conclusions drawn in Section 5.

\section{Model problem and a physics-informed neural network solution}
\subsection{Model problem of elasticity}

Consider an elastic body whose undeformed configuration $\boldsymbol{X}$ is defined in a domain $\Omega\subset \mathbb{R}^d$ ($d=1,2,3$) with a Lipschitz boundary $\Gamma=\Gamma_D\cup \Gamma_N$, where $\Gamma_D\neq \emptyset$ is the Dirichlet boundary, $\Gamma_N$ is the Neumann boundary, and $\Gamma_D\cap\Gamma_N=\emptyset$. The elastic body is subjected to a prescribed body force of density $\boldsymbol{f}\in [L^2(\Omega)]^d$ in $\Omega$ with respect to the undeformed volume, a prescribed displacement $\boldsymbol{u}_D$ on $\Gamma_D$, and a prescribed surface force of density $\boldsymbol{t}\in [L^2(\Gamma_N)]^d$ with respect to the undeformed surface area. An elasticity problem is seek to find the the vector field of displacements $\boldsymbol{u}:\Omega \to \mathbb{R}^d$ and the stress tensor field $\boldsymbol{\sigma}: \Omega \to \mathbb{R}^{d\times d}$ that satisfy (i) the compatibility condition: 
\begin{equation}\label{comp}
\boldsymbol{u}\in [H_1(\Omega)]^d\quad \text{and}\quad \boldsymbol{u}=\boldsymbol{u}_D\quad \mathrm{on}\;\Gamma_D\,,
\end{equation}
i.e. the continuity of the displacement field $\boldsymbol{u}$ and  the Dirichlet boundary condition; (ii) the equilibrium condition:
\begin{equation}\label{equi}
\mathrm{div}\boldsymbol{\sigma}+\boldsymbol{f}=\boldsymbol{0}\quad\text{and}\quad \mathrm{in}\;\Omega\,,\quad \boldsymbol{\sigma}\boldsymbol{n}=\boldsymbol{t}\quad \mathrm{on}\;\Gamma_N\,,
\end{equation}
i.e. the momentum equation and the Neumann boundary condition; as well as (iii) the elastic constitutive relation:
\begin{equation}\label{cr}
\boldsymbol{\sigma}=\mathcal{C}[\nabla\boldsymbol{u}]\,,
\end{equation}
in which $\mathcal{C}: \mathbb{R}^{d\times d}\to \mathbb{R}^{d\times d}$ denotes the constitutive relation that maps the displacement gradient $\nabla\boldsymbol{u}$ to the stress tensor $\boldsymbol{\sigma}$. Note that $\boldsymbol{\sigma}$ is taken as the Cauchy stress tensor in linear elasticity under small deformation while more generally considered as the first Piola-Kirchhoff stress tensor in hyperelasticity under finite deformation. The constitutive relation $\mathcal{C}$ can be expressed in the following form in hyperelasticity:
\begin{equation}
\boldsymbol{\sigma} = \partial W(\nabla\boldsymbol{u})/\partial (\nabla\boldsymbol{u})\,,
\end{equation}
where $W:\mathbb{R}^{d\times d}\to \mathbb{R}$ is the potential energy density of the considered material. As a special case of hyperelasticity, linear elasticity has the potential energy density in the form of $W = \frac{1}{2}(\nabla\boldsymbol{u}):\boldsymbol{K}:(\nabla\boldsymbol{u})$, $\boldsymbol{K}$ being the Hooke's stiffness tensor, and the constitutive relation is the Hooke's law $\boldsymbol{\sigma} = \boldsymbol{K}:\nabla\boldsymbol{u}$.

Moreover, for the follow-up discussions on error bounding, the complementary energy density $W^*:\mathbb{R}^{d\times d}\to \mathbb{R}, \vb*{\tau}\mapsto W^*(\vb*{\tau})$ is introduced as the Legendre transformation of $W$, i.e.,
\begin{equation}\label{lt}
W^*(\vb*{\tau})=\sup_{\vb*{\epsilon}\in \mathbb{R}^{d\times d}}\left\{ \vb*{\tau}:\vb*{\epsilon} - W(\vb*{\epsilon})  \right\}\,.
\end{equation}
Assuming that the potential energy density $W$ is a convex function, the constitutive relation \eqref{cr} can thus be alternatively written as
\begin{equation}
W(\nabla \vb*{u}) + W^*(\vb*{\sigma}) - \vb*{\sigma}:\nabla \vb*{u} =0\,.
\end{equation}

\subsection{A mixed form of PINN solutions}
In this work, we use separate neural networks to approximate the components of both the displacement field $\vb*{u}$ and the stress field $\vb*{\sigma}$. Taking the case of $d=2$ as an example, one has
\begin{equation}
\vb*{u}(\vb*{X})\simeq \vb*{u}^\texttt{NN}(\vb*{X})=\begin{bmatrix}
u_x^\texttt{NN}(\vb*{X})\\ u_y^\texttt{NN}(\vb*{X})
\end{bmatrix}\,,\quad
\vb*{\sigma}(\vb*{X})\simeq \vb*{\sigma}^\texttt{NN}(\vb*{X})=\begin{bmatrix}
\sigma_{xx}^\texttt{NN}(\vb*{X}) & \sigma_{xy}^\texttt{NN}(\vb*{X}) \\ \sigma_{yx}^\texttt{NN}(\vb*{X}) & \sigma_{yy}^\texttt{NN}(\vb*{X}) 
\end{bmatrix}\,.
\end{equation}
Note that $\sigma_{xy}= \sigma_{yx}$ for linear elasticity. To train these neural networks informed by the governing equations \eqref{comp}, \eqref{equi}, and \eqref{cr}, a loss function $\mathcal{L}$ can be formulated as
\begin{equation}\label{MSE}
\begin{split}
\mathcal{L}= &~  \texttt{MSE}_{\Gamma_D} + \texttt{MSE}_{\vb*{f}} +  \texttt{MSE}_{\Gamma_N} + \eta \texttt{MSE}_{\mathcal{C}} + \alpha(\texttt{MSE}_{\vb*{u}}+\texttt{MSE}_{\vb*{\sigma}})\,,\\
\texttt{MSE}_{\Gamma_D} = & ~ \frac{1}{N_{\Gamma_D}} \sum_{k=1}^{N_{\Gamma_D}}\|\vb*{u}^\texttt{NN }(\vb*{X}_k|_{\Gamma_D}; \vb*{\theta})-\vb*{u}_D(\vb*{X}_k|_{\Gamma_D})\|^2_2\,,\\
\texttt{MSE}_{\vb*{f}} = & ~ \frac{1}{N_\Omega} \sum_{k=1}^{N_\Omega}\|\mathrm{div}\vb*{\sigma}^\texttt{NN}(\vb*{X}_k|_{\Omega})+\vb*{f}(\vb*{X}_k|_{\Omega})\|_F^2\,,\\
\texttt{MSE}_{\Gamma_N} = & ~ \frac{1}{N_{\Gamma_N}} \sum_{k=1}^{N_{\Gamma_N}}\|\vb*{\sigma}^\texttt{NN}(\vb*{X}_k|_{\Gamma_N})\vb*{n}-\vb*{t}(\vb*{X}_k|_{\Gamma_N})\|_2^2\,,\\
\texttt{MSE}_{\mathcal{C}} =  ~ & \frac{1}{N_\Omega}\sum_{k=1}^{N_\Omega}\|\vb*{\sigma}^\texttt{NN}(\vb*{X}_k|_{\Omega})-\mathcal{C}[\nabla\vb*{u}^\texttt{NN}(\vb*{X}_k|_{\Omega})]\|_F^2\,,\\
\texttt{MSE}_{\vb*{u}} = & ~\frac{1}{N_\Omega}\sum_{k=1}^{N_\Omega}\|\vb*{u}^\texttt{NN}(\vb*{X}_k|_\Omega)-\vb*{u}(\vb*{X}_k|_\Omega)\|_2^2\,,\\
\texttt{MSE}_{\vb*{\sigma}} =  ~ & \frac{1}{N_\Omega}\sum_{k=1}^{N_\Omega}\|\vb*{\sigma}^\texttt{NN}(\vb*{X}_k|_\Omega)-\vb*{\sigma}(\vb*{X}_k|_\Omega)\|_F^2\,.
\end{split}
\end{equation}
where $\{\vb*{X}_1|_\Omega,\cdots,\vb*{X}_{N_\Omega}|_\Omega\}$, $\{\vb*{X}_1|_{\Gamma_N},\cdots,\vb*{X}_{N_{\Gamma_N}}|_{\Gamma_N}\}$ and $\{\vb*{X}_1|_{\Gamma_D},\cdots,\vb*{X}_{N_{\Gamma_D}}|_{\Gamma_D}\}$ are the collocation points over the domain $\Omega$ and those along the boundaries $\Gamma_N$ and $\Gamma_D$, respectively. $\eta\in \mathbb{R}^+$ is the penalty coefficient for imposing the constitutive relation \eqref{cr}. $\alpha$ is a 0-1 binary variable: the case  $\alpha=0$ is a pure boundary value problem (BVP), while the case $\alpha=1$ considers the enhancement by the collocation data of $\vb*{u}$ and $\vb*{\sigma}$ over $\Omega$, i.e., the formulation \eqref{MSE} gives a physics-informed regression with the PDE constraints \eqref{comp}, \eqref{equi} and \eqref{cr} when $\alpha = 1$ .

\section{An energy-based error bound}

\subsection{Constitutive relation error}

\begin{definition}[CRE]
	A constitutive relation error (CRE) \cite{ladeveze2005mastering} is a functional $\Psi:KA\times SA\to\mathbb{R}$ defined as follows:
	\begin{equation}\label{CREhyp}
	\Psi(\hat{\vb*{u}},\hat{\vb*{\sigma}})=\int_\Omega \left[W(\nabla\hat{\vb*{u}})+W^*(\hat{\vb*{\sigma}})- \hat{\vb*{\sigma}}:\nabla\hat{\vb*{u}}\right]\dd\Omega \,,\quad  (\hat{\vb*{u}},\hat{\vb*{\sigma}})\in KA\times SA\,,
	\end{equation}
	where $KA = \{\hat{\vb*{u}}\in [H^1(\Omega)]^d: \hat{\vb*{u}}(\Gamma_D)=\vb*{u}_D~\text{a.e.}\}$ is the set of all kinematically admissible solutions that satisfy the compatibility condition \eqref{comp}, and $SA=\{\hat{\vb*{\sigma}}\in [H(\mathrm{div}, \Omega)]^d: \mathrm{div}\hat{\vb*{\sigma}}+ \vb*{f}=\vb{0}~ \text{a.e. in }\Omega, ~ \hat{\vb*{\sigma}}\vb*{n}=\vb*{t}~\text{a.e. on }\Gamma_N\}$ is the set of all statically admissible solutions that satisfy the equilibrium condition \eqref{equi}.
\end{definition}

A natural result of the Legendre transformation \eqref{lt} is the Fenchel-Young inequality $W(\vb*{\epsilon})+W^*(\vb*{\tau})-\vb*{\tau}:\vb*{\epsilon}\geq 0$, $\forall (\vb*{\epsilon},\vb*{\tau})\in \mathbb{R}^{d\times d}\times \mathbb{R}^{d\times d}$, which further gives the following proposition.

\begin{proposition}
	$\Psi(\hat{\vb*{u}},\hat{\vb*{\sigma}})\geq 0$, $ \forall (\hat{\vb*{u}},\hat{\vb*{\sigma}})\in KA\times SA$. 
\end{proposition}

To evaluate the errors between the admissible fields $ (\hat{\vb*{u}},\hat{\vb*{\sigma}})\in KA\times SA$ and the exact fields $(\vb*{u},\vb*{\sigma})$, two error functionals are defined as follows.

\begin{definition}[Error functionals]\label{errorfunc}
	Error functionls $\phi: [H^1(\Omega)]^d \to \mathbb{R}$ and $\varphi : [H(\mathrm{div}, \Omega)]^d \to \mathbb{R}$ are introduced in the following form:
	\begin{equation}\label{errdef}
	\begin{split}
	& \phi(\vb*{e})= \int_\Omega\left[ W(\nabla(\vb*{u}+\vb*{e}))-W(\nabla\vb*{u})-\vb*{\sigma}:\nabla\vb*{e}\right]\dd\Omega\,,\quad  \boldsymbol{e}\in [H^1(\Omega)]^d\,,\\
	& \varphi(\vb*{r})=\int_\Omega\left[W^*(\vb*{\sigma}+\vb*{r})-W^*(\vb*{\sigma})-\vb*{r}:\nabla\vb*{u} \right]\dd\Omega\,,\quad  \vb*{r}\in [H(\mathrm{div}, \Omega)]^d\,,
	\end{split}
	\end{equation}
	where $(\vb*{u},\vb*{\sigma})$ is the exact solution pair of the elasticity problem.
\end{definition}

\begin{proposition}
	The error functionals $\phi$ and $\varphi$ have the following properties:
	\begin{equation}\label{assum}
	\phi(\hat{\vb*{u}}-\vb*{u})\geq 0\,,\quad \varphi(\hat{\vb*{\sigma}}-\vb*{\sigma})\geq 0\,,\quad \forall (\hat{\vb*{u}},\hat{\vb*{\sigma}})\in KA\times SA\,.
	\end{equation}
\end{proposition}

\noindent \emph{Proof.} It can be verified that $\phi(\hat{\vb*{u}}-\vb*{u})=\Psi(\hat{\vb*{u}},\vb*{\sigma})-\Psi({\vb*{u}},\vb*{\sigma})=\Psi(\hat{\vb*{u}},\vb*{\sigma})$ and $\varphi(\hat{\vb*{\sigma}}-\vb*{\sigma})=\Psi(\vb*{u},\hat{\vb*{\sigma}})-\Psi(\vb*{u},{\vb*{\sigma}})=\Psi(\vb*{u},\hat{\vb*{\sigma}})$, both being not less than zero due to Proposition 3.1. The equality holds true if and only if $\hat{\vb*{u}}={\vb*{u}}$ and $\hat{\vb*{\sigma}}={\vb*{\sigma}}$, respectively. ~~~~~$\square$

\begin{proposition}
	The CRE and the error functionals in Definition \ref{errorfunc} are interlinked by the following identity:
	\begin{equation}\label{hypsplit}
	\Psi(\hat{\boldsymbol{u}},\hat{\boldsymbol{\sigma}})=  \phi(\hat{\boldsymbol{u}}-\boldsymbol{u})+\varphi(\hat{\boldsymbol{\sigma}}-\boldsymbol{\sigma})\geq
	\begin{cases}
	\phi(\hat{\boldsymbol{u}}-\boldsymbol{u})\,,\\
	\varphi(\hat{\boldsymbol{\sigma}}-\boldsymbol{\sigma})\,,
	\end{cases}
	\quad \forall (\hat{\boldsymbol{u}},\hat{\boldsymbol{\sigma}})\in KA\times SA\,,
	\end{equation}
	i.e., the CRE can be represented as the sum of two error functionals of the admissible solutions.
\end{proposition}

\noindent \emph{Proof.} From the definitions of error functionals \eqref{errdef} and that of the CRE \eqref{CREhyp}, one has
\begin{equation*}
\begin{split}
&~ \phi(\hat{\boldsymbol{u}}-\boldsymbol{u})+\varphi(\hat{\boldsymbol{\sigma}}-\boldsymbol{\sigma}) \\
= &~\int_\Omega \left[ \left(W(\nabla\hat{\vb*{u}})+W^*(\hat{\vb*{\sigma}})\right)- \left(W(\nabla\vb*{u})+W^*({\vb*{\sigma}})\right)
-\vb*{\sigma}:\nabla(\hat{\vb*{u}}-\vb*{u}) - (\hat{\vb*{\sigma}}-\vb*{\sigma}):\nabla\vb*{u} \right]\dd\Omega \\
= &~\left(\Psi(\hat{\boldsymbol{u}},\hat{\boldsymbol{P}})+\int_\Omega \hat{\vb*{\sigma}}:\nabla\hat{\vb*{u}}~\dd\Omega \right) - \int_\Omega \vb*{
	\sigma}:\nabla \vb*{u} ~\dd \Omega -\int_\Omega \left[\vb*{\sigma}:\nabla(\hat{\vb*{u}}-\vb*{u}) + (\hat{\vb*{\sigma}}-\vb*{\sigma}):\nabla\vb*{u} \right]~\dd\Omega \\
= &~\Psi(\hat{\boldsymbol{u}},\hat{\boldsymbol{P}}) + \int_\Omega  (\hat{\vb*{\sigma}}-\vb*{\sigma}): \nabla (\hat{\vb*{u}}-\vb*{u})~\dd\Omega\,.
\end{split}
\end{equation*}
Green's formula gives that
\begin{equation}\label{Green}
\int_\Omega\boldsymbol{\tau}:\nabla \boldsymbol{v}~\dd\Omega=\int_{\Omega}-\mathrm{div}\boldsymbol{\tau}\cdot\boldsymbol{v}\dd{\Omega}+\int_{\Gamma}(\boldsymbol{\tau}\boldsymbol{n})\cdot\boldsymbol{v}\dd{\Gamma}\,,\quad (\boldsymbol{v},\boldsymbol{\tau})\in [H_1(\Omega)]^d \times [H(\mathrm{div},\Omega)]^d\,.
\end{equation}
Taking $\vb*{\tau}=\hat{\vb*{\sigma}}-\vb*{\sigma}$ and $\vb*{v}=\hat{\vb*{u}}-\vb*{u}$ yields that $\int_\Omega  (\hat{\vb*{\sigma}}-\vb*{\sigma}): \nabla (\hat{\vb*{u}}-\vb*{u})~\dd\Omega =0$, since $\mathrm{div}(\hat{\vb*{\sigma}}-\vb*{\sigma})=\vb{0}$ in $\Omega$, $(\hat{\vb*{\sigma}}-\vb*{\sigma})\vb*{n}=\vb{0}$ on $\Gamma_N$, and $\hat{\vb*{u}}-\vb*{u}=\vb{0}$ on $\Gamma_D$. Then \eqref{hypsplit} is obtained in view of the fact that $\phi(\hat{\boldsymbol{u}}-\boldsymbol{u})\geq 0$ and $\varphi(\hat{\boldsymbol{\sigma}}-\boldsymbol{\sigma})\geq 0$. ~~~~~$\square$

In linear elasticity under small deformation, the difference between deformed and undeformed configurations is not taken into consideration, and the first Piola-Kirchhoff stress tensor coincides with the Cauchy stress tensor. The potential and complementary energy densities are written as $W(\nabla \vb*{u})=\frac{1}{2}\nabla \vb*{u}:\vb*{K}:\nabla \vb*{u}$ and $W^*(\vb*{\sigma})=\frac{1}{2}\vb*{\sigma}:\vb*{K}^{-1}:\vb*{\sigma}$, respectively. The error functionals for both the displacement and stress fields are formulated as $\phi(\vb*{e})=\int_\Omega \frac{1}{2}\nabla \vb*{e}:\vb*{K}:\nabla \vb*{e}~\dd\Omega$ and $\varphi(\vb*{r})=\int_\Omega \frac{1}{2}\vb*{r}:\vb*{K}^{-1}:\vb*{r}~\dd\Omega$, respectively, $(\vb*{e},\vb*{r})\in (KA-\vb*{u})\times (SA-\vb*{\sigma})$. Furthermore, the CRE has a second-order form $\Psi(\hat{\boldsymbol{u}},\hat{\boldsymbol{\sigma}}) = \int_\Omega \frac{1}{2}(\hat{\vb*{\sigma}}-\vb*{K}:\nabla\hat{\vb*{u}}):\vb*{K}^{-1}:(\hat{\vb*{\sigma}}-\vb*{K}:\nabla\hat{\vb*{u}})~\dd\Omega$, $(\hat{\vb*{u}},\hat{\vb*{\sigma}})\in KA\times SA$. Moreover, we refer to \cite{guo2016legendre} for a more generalized discussion on CREs in convex problems.

\subsection{Bounding global errors in PINN solutions}

When $\eta$ is set to be a small value, i.e., $\eta \ll 1$, the compatibility constraint \eqref{comp} and the equilibrium condition \eqref{equi} are enforced on $\vb*{u}^\texttt{NN}$ and $\vb*{\sigma}^\texttt{NN}$, respectively, in the sense of minimizing the corresponding penalty terms in $\mathcal{L}$. By taking $\hat{\vb*{u}}\simeq\vb*{u}^\texttt{NN}$ and $\hat{\vb*{\sigma}}\simeq\vb*{\sigma}^\texttt{NN}$, the CRE $\Psi$ provides an error bound for both $\vb*{u}^\texttt{NN}$ and $\vb*{\sigma}^\texttt{NN}$, i.e.,
\begin{equation}\label{upperbd}
\phi(\vb*{u}^\texttt{NN}-\vb*{u})\lesssim \Psi(\vb*{u}^\texttt{NN},\vb*{\sigma}^\texttt{NN}), \quad \text{and}\quad \varphi(\vb*{\sigma}^\texttt{NN}-\vb*{\sigma})\lesssim \Psi(\vb*{u}^\texttt{NN},\vb*{\sigma}^\texttt{NN})\,,
\end{equation}
where the global discretization errors in the neural network solutions $\vb*{u}^\texttt{NN}$ and $\vb*{\sigma}^\texttt{NN}$ are measured by the functionals $\phi$ and $\varphi$, respectively. From the perspective of neural network training, $\phi$ and $\varphi$ are adopted to quantify the global generalization errors in networks $\vb*{u}^\texttt{NN}$ and $\vb*{\sigma}^\texttt{NN}$, respectively. As the solutions $(\vb*{u}^\texttt{NN},\vb*{\sigma}^\texttt{NN})$ obtained by minimizing the loss function $\mathcal{L}$ only approximately satisfy the compatibility constraint \eqref{comp} and the equilibrium condition \eqref{equi}, we use the symbol $\lesssim$ for the bounding property \eqref{upperbd} in an approximate sense.

In the CRE estimation for the displacement-based, conforming finite element method, it is required to make additional efforts to construct a stress field that satisfies the equilibrium equation \eqref{equi}. Such stress recovery techniques \cite{ladeveze2005mastering,pled2011techniques,gallimard2009constitutive} are usually based on element-wise construction and implemented at a high computational price. However, the mixed form of PINNs gives the admissible solution pair $(\vb*{u}^\texttt{NN},\vb*{\sigma}^\texttt{NN})$ simultaneously, then the CRE can be directly applied without additional computational efforts. When a larger amount of training data or a more sophisticated neural network architecture is adopted, the accuracy of $\vb*{u}^\texttt{NN}$ and $\vb*{\sigma}^\texttt{NN}$ can be improved simultaneously through minimizing the loss function $\mathcal{L}$ in \eqref{MSE}, and such asymptotic behavior can be indicated by the decay of the CRE value $\Psi$ in the meantime.

\section{Numerical example}

To demonstrate the proposed error bounds, we consider a square plate defined in the domain $\Omega = ]0,1[^2$ with unit side length $L_x=L_y=1$, introduced in \cite{haghighat2020deep}. Let us identify different sides of the square with $\Gamma_{x^-}=\{\vb*{X}=(x,y)^\text{T}| x=0, y\in[0,1]\}$, $\Gamma_{x^+}=\{\vb*{X}=(x,y)^\text{T}| x=1, y\in[0,1]\}$, $\Gamma_{y^-}=\{\vb*{X}=(x,y)^\text{T}| y=0, x\in[0,1]\}$, and $\Gamma_{y^+}=\{\vb*{X}=(x,y)^\text{T}| y=1, x\in[0,1]\}$. The plate is subjected to the following boundary conditions: (a) $\sigma_{xx}=0$ and $u_y=0$ on $\Gamma_x^-$ and $\Gamma_x^+$, (b) $u_x=u_y=0$ on $\Gamma_y^-$, and (c) $u_x=0$ and $\sigma_{yy}=(\lambda + 2\mu)Q\sin(\pi x)$ on $\Gamma_y^+$. Additionally, the plate is subjected to the following body forces:
\begin{equation*}
\begin{split}
f_x(x,y) = & \lambda \left[4\pi^2 \cos(2\pi x) \sin(\pi y) - \pi \cos(\pi x) Q y^3 \right]  \\
& + \mu \left[9\pi^2 \cos(2\pi x) \sin(\pi y) - \pi \cos(\pi x) Q y^3 \right], \\
f_y(x,y) = & \lambda \left[-3\sin(\pi x) Qy^2 +2\pi^2 \sin(2\pi x) \cos(\pi y) \right] \\
& + \mu \left[-6\sin(\pi x) Qy^2 +2\pi^2\sin(2\pi x)\cos(\pi y) + \pi^2\sin(\pi x) Q y^4/4 \right]. 
\end{split}
\end{equation*}
Here $\lambda$ and $\mu$ are the two Lamé parameters of a homogeneous and isotropic material, and we take $Q = 4$. The analytical solution to this problem is given as 
\begin{equation}
u_x(x,y) =\cos(2\pi x) \sin(\pi y) \quad \text{and}\quad u_y(x,y)=\sin(\pi x) Qy^4/4
\end{equation}
and plotted in Fig. \ref{fig:solution}.

\begin{figure}[H]
	\centering
	\includegraphics[width=1\textwidth]{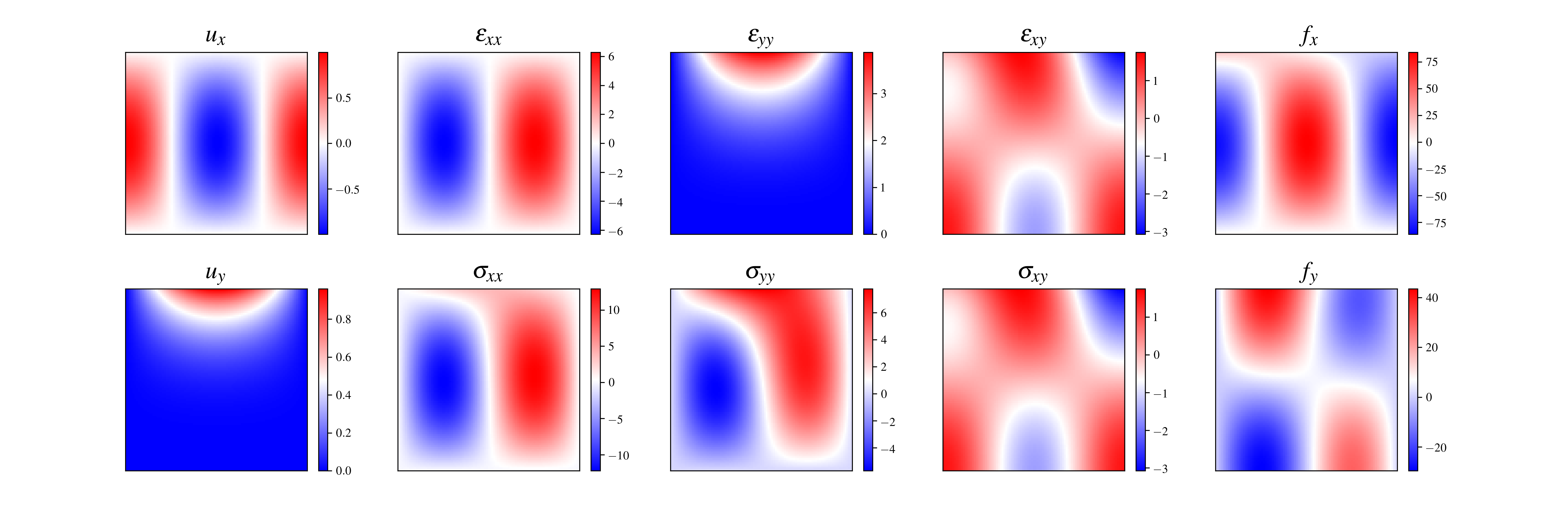}
	\caption{The exact displacement components $(u_x, u_y)$, strain components $(\varepsilon_{xx}, \varepsilon_{yy}, \varepsilon_{xy})$, and stress components $(\sigma_{xx}, \sigma_{yy}, \sigma_{xy})$ for the linear elasticity problem, subjected to the body forces $(f_x, f_y)$. }
	\label{fig:solution}
\end{figure}

In this example, the PINN framework is employed in the following two scenarios: 
\begin{itemize}
	\item \textbf{scenario (a): PINN as a solution method}, i.e., $\alpha = 0$. The problem is considered as a boundary value problem, where data are given at the boundaries in the form of displacement or stress conditions. A sampling grid of $100\times 100$ is used for the collocation of equilibrium equations and constitutive relations, except when we change the sampling grid to investigate the corresponding asymptotic behaviors. Moreover, we take $\eta = 0.01$ in this scenario.
	\item \textbf{scenario (b): PINN as a regression method}, i.e., $\alpha = 1$. It is assumed that the displacement and stress data are given at uniformly distributed grid points, and the objective is to construct a regression model that also satisfies the underlying physics of linear elasticity. Here the data are given on a uniform grid of $40 \times 40$, except when we change the sampling grid to investigate the corresponding asymptotic behaviors. Since the grid data for both displacements and stresses already impose very strong constraints on the solution fields, we take $\eta=0$ in this scenario, i.e., loosing the constraints by constitutive relations.
\end{itemize}

There are three types of errors associated with the training of neural networks, including the optimization error, the generalization error, and the approximation error, as discussed in \cite{shin2020convergence}. Here, we quantify the global generalization errors in the displacement and stress fields by $\phi(\vb*{u}^\texttt{NN}-\vb*{u})$ and $\varphi(\vb*{\sigma}^\texttt{NN}-\vb*{\sigma})$, respectively. All these integrals are approximately evaluated on a fine \emph{test} grid of $200 \times 200$. We consider the following three cases:

\begin{itemize}
	\item \textbf{Generalization error in sampling}. As the first case, we quantify the generalization errors in the displacement and stress fields with respect to the number of sampling (collocation) points used for optimizing the loss function. Here, we adopt a fixed size neural network with 4 hidden layers and 20 neurons in each layer to construct all the approximate solutions $u_x^\texttt{NN}$,  $u_y^\texttt{NN}$, $\sigma_{xx}^\texttt{NN}$, $\sigma_{yy}^\texttt{NN}$, and $\sigma_{xy}^\texttt{NN}$. We then use seven grids of sizes $40\times40$, $50\times 50$, $60\times60$, $70\times 70$, $80\times80$, $90\times 90$, and $100\times100$ to sample the loss function during the optimization and train the neural networks. 
	
	\item \textbf{Generalization error in neurons}. As the second case, we study the generalization errors in the displacement and stress solutions with respect to the number of neurons per each layer. For this purpose, we fix the number of hidden layers of each network to 4 and change the width of the network as 20, 30, 40, 50, 60, 70, and 80. 
	
	\item \textbf{Generalization error in layers}. As the last case, we study the generalization errors in the displacement and stress solutions with respect to the number of hidden layers. Therefore, we adopt a fixed number of neurons for each hidden layer as 20 and vary the number of hidden layers from 2 to 10 in scenario (a), while 4 to 10 in scenario (b). 
\end{itemize}

The comparisons of global discretization errors $\phi(\vb*{u}^\texttt{NN}-\vb*{u})$,  $\varphi(\vb*{\sigma}^\texttt{NN}-\vb*{\sigma})$ and the CRE $\Psi(\vb*{u}^\texttt{NN},\vb*{\sigma}^\texttt{NN})$ are shown in Figs. \ref{fig:bvpresults} and \ref{fig:rgnresults} for the scenarios (a) and (b), respectively. All results are compiled on a \emph{test} grid of size $200\times200$, with which the integrals are approximately calculated. Note that this distinguishes from the classical finite element method in which the results are analyzed at the same locations as they are evaluated, i.e., at the nodes or quadrature points. In all the cases considered in this example, it is verified that the CRE guarantees an upper bound of the global generalization errors in both the displacement and stress solutions by PINNs. As seen in the stacked plots, the equality $\phi(\vb*{u}^\texttt{NN}-\vb*{u})+\varphi(\vb*{\sigma}^\texttt{NN}-\vb*{\sigma}) = \Psi(\vb*{u}^\texttt{NN},\vb*{\sigma}^\texttt{NN})$ is often not strictly satisfied and there exists a relative error within $15\%$. The reason mainly lies in the fact that the solution fields $(\vb*{u}^\texttt{NN},\vb*{\sigma}^\texttt{NN})$ obtained by minimizing $\mathcal{L}$ cannot strictly satisfy the corresponding admissible conditions, thus the term $\int_\Omega(\vb*{\sigma}^\texttt{NN}-\vb*{\sigma}):\nabla(\vb*{u}^\texttt{NN}-\vb*{u})~\dd\Omega$ cannot vanish (see the proof of Proposition 3.3). In spite of this, the CRE still effectively provides an upper bound of the global discretization errors and can be used as an indicator for the credibility of PINN solutions. 

In general, scenario (a) can be considered as an unsupervised learning approach, in which no training data is provided inside the domain, while scenario (b) is of the supervised learning class. The supervised learning method is well understood. With more training points, one can improve the accuracy of the predictions up to a limit depending on the network size. This is indeed what we can observe in the scenario (b), see Fig. \ref{rgnsample2}. The use of more training points reduces the generalization errors of the PINN model, and high accuracy can be achieved with very fine training grids. The width and depth of the network also show more predictable behaviour as a result of supervised learning, as shown in Figs. \ref{rgnneuron2} and \ref{rgnlayer2}. On the other hand, without imposing any data inside the domain, solving a PDE problem only by minimizing a loss function that incorporates the governing equations is an unsupervised learning task. This task is generally much harder and requires much more training time and iterations. We find that after a relatively small number of sampling points, the accuracy of predictions by a fixed network architecture can hardly be improved by adding more training points, see Fig. \ref{bvpsample2}. Further improvements demand a significant fine-tuning of the optimizer and many training epochs, which we try to avoid here. Such a behavior of PINN as a BVP solution method is also consistent with the observations made in other studies \cite{shin2020convergence}. In this case, the bounding property of the proposed CRE estimator is still guaranteed. Since the true errors are not computable, the CRE estimator can be evaluated for the observation of convergence behavior.

\begin{figure}[ht]
	\centering
	\subfigure[ $\phi/\Psi$ and $\varphi/\Psi$ versus training sample grid]{\includegraphics[width=.48\textwidth]{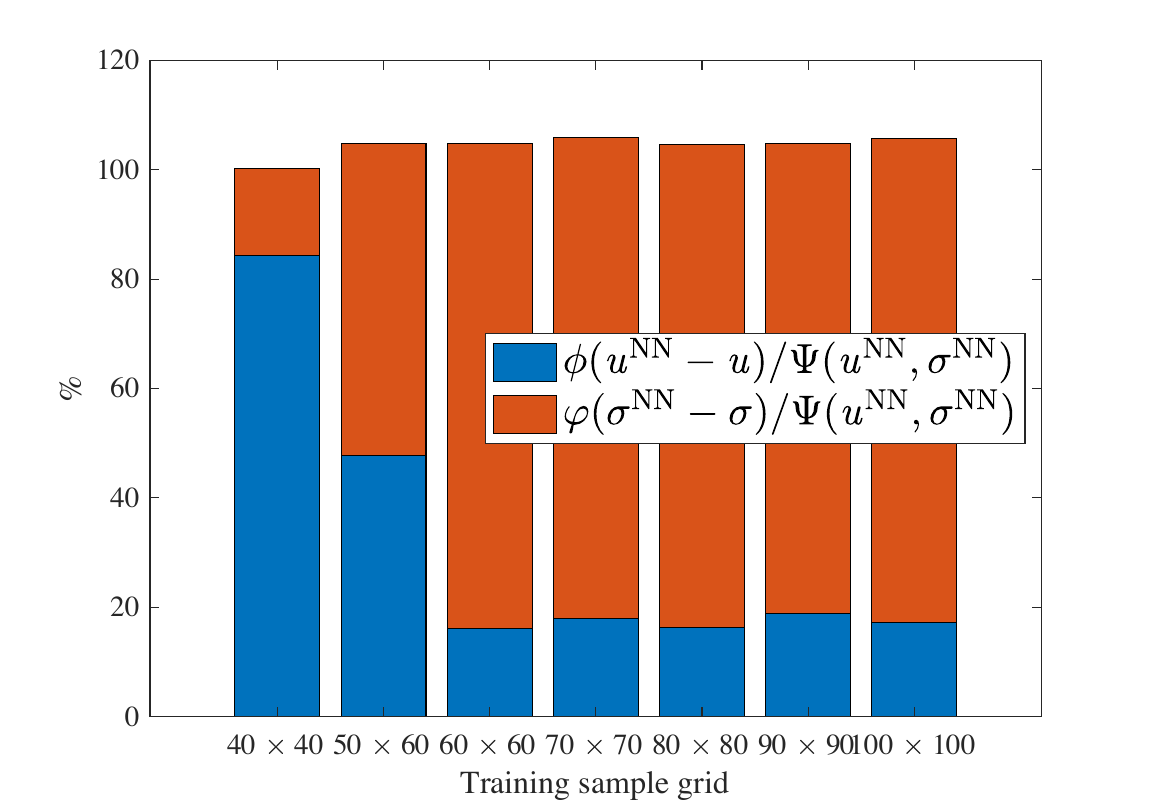}}
	\hfill
	\subfigure[Global generalization errors and CRE versus training sample grid]{\includegraphics[width=.48\textwidth]{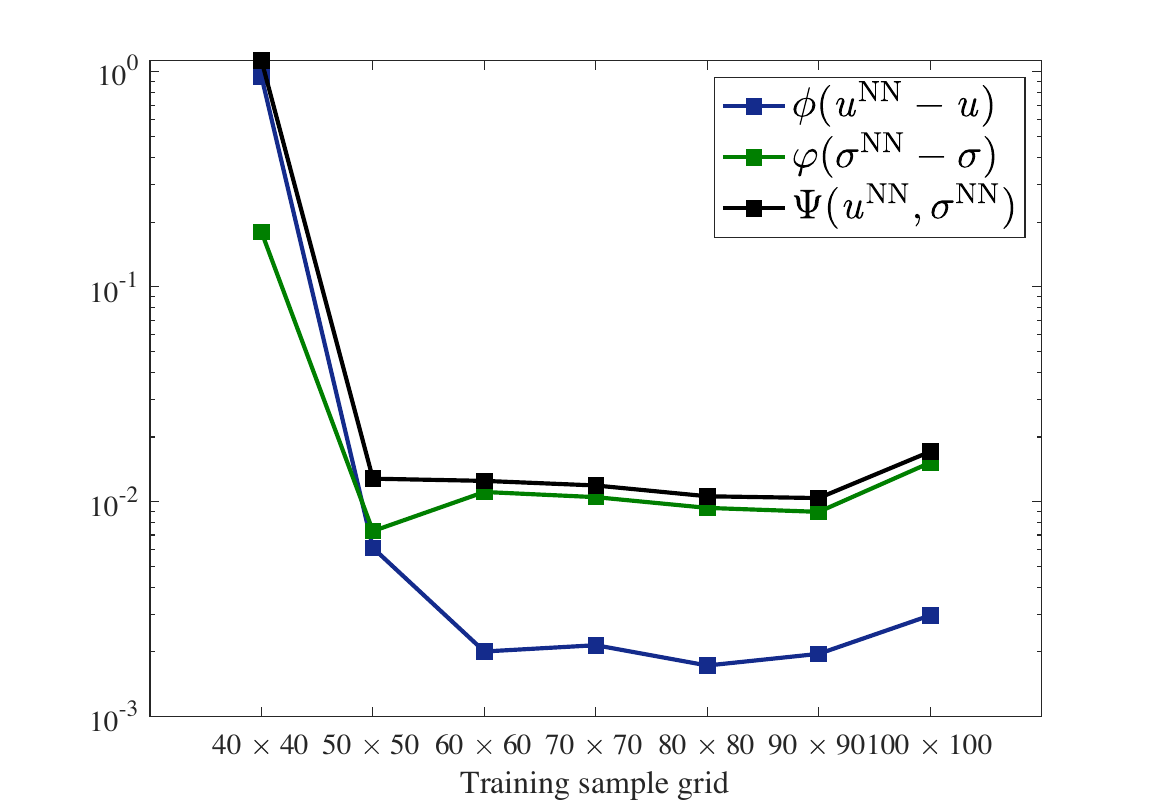}\label{bvpsample2}}

	\subfigure[$\phi/\Psi$ and $\varphi/\Psi$ versus number of neurons]{\includegraphics[width=.48\textwidth]{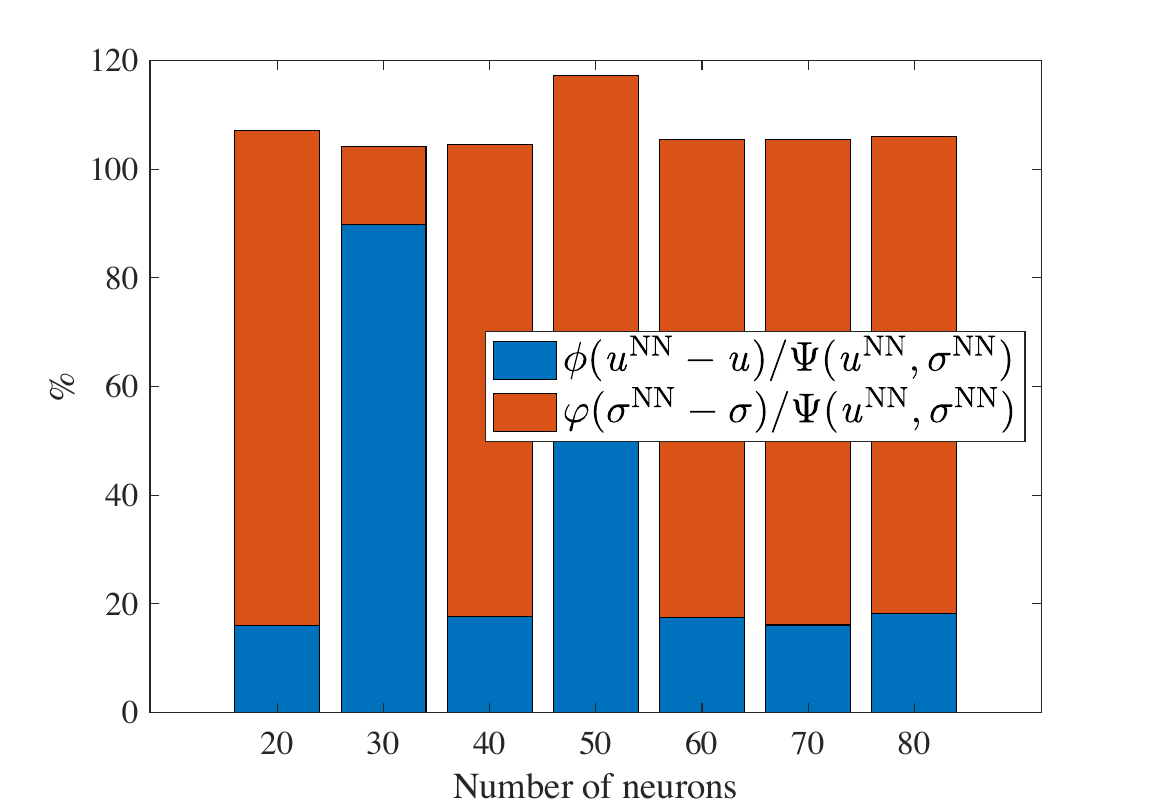}}
	\hfill
	\subfigure[Global generalization errors and CRE versus number of neurons]{\includegraphics[width=.48\textwidth]{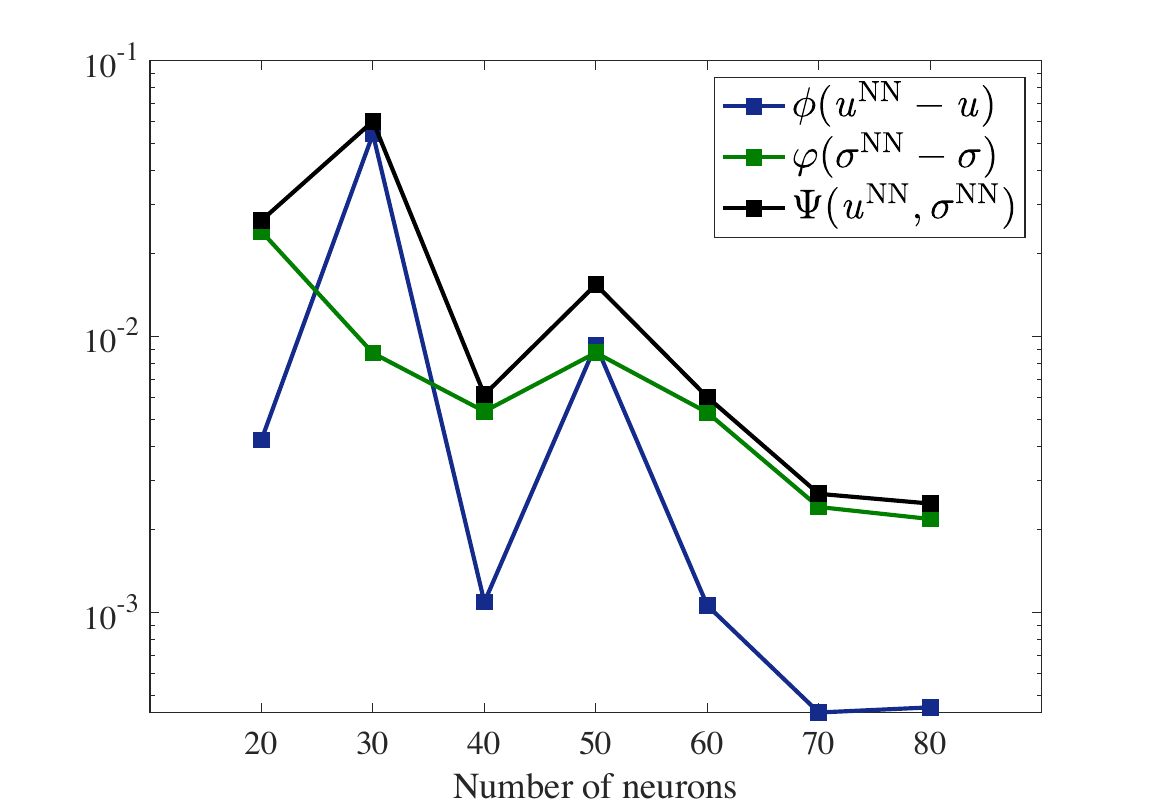}}

	\subfigure[$\phi/\Psi$ and $\varphi/\Psi$ versus number of layers]{\includegraphics[width=.48\textwidth]{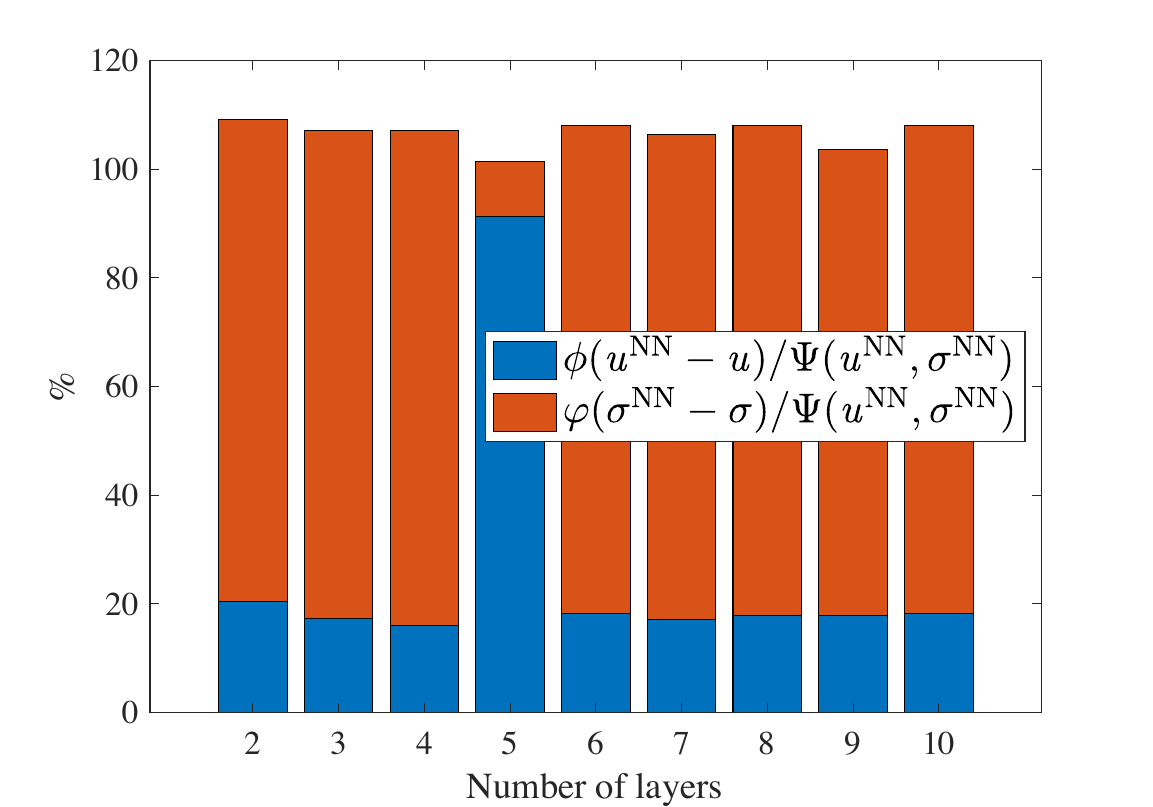}}
	\hfill
	\subfigure[Global generalization errors and CRE versus number of layers]{\includegraphics[width=.48\textwidth]{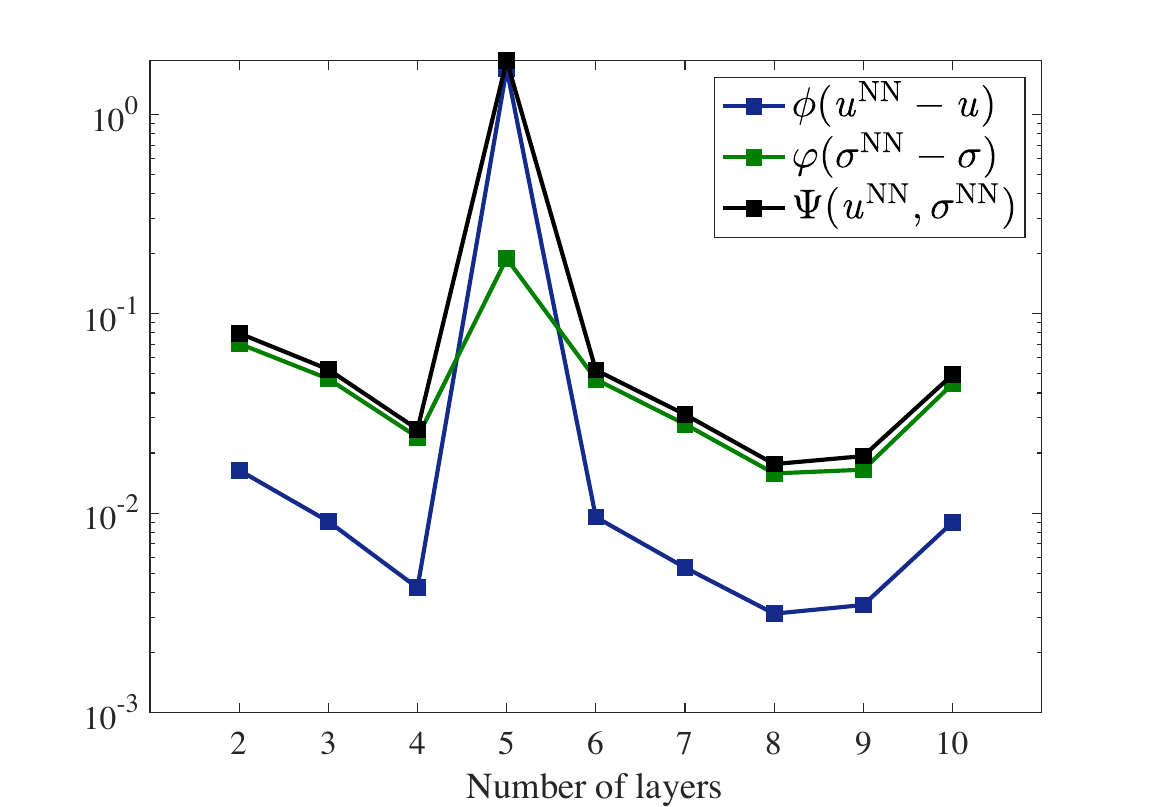}}

	\caption{Comparison of the energy error in displacements $\phi(\vb*{u}^\texttt{NN}-\vb*{u})$, the energy error in stresses $\varphi(\vb*{\sigma}^\texttt{NN}-\vb*{\sigma})$, and the CRE $\Psi(\vb*{u}^\texttt{NN},\vb*{\sigma}^\texttt{NN})$ in the scenario (a) where PINN is used as a BVP solution method.}
	\label{fig:bvpresults}
\end{figure}

\begin{figure}[ht]
	\centering
	\subfigure[ $\phi/\Psi$ and $\varphi/\Psi$ versus training sample grid]{\includegraphics[width=.48\textwidth]{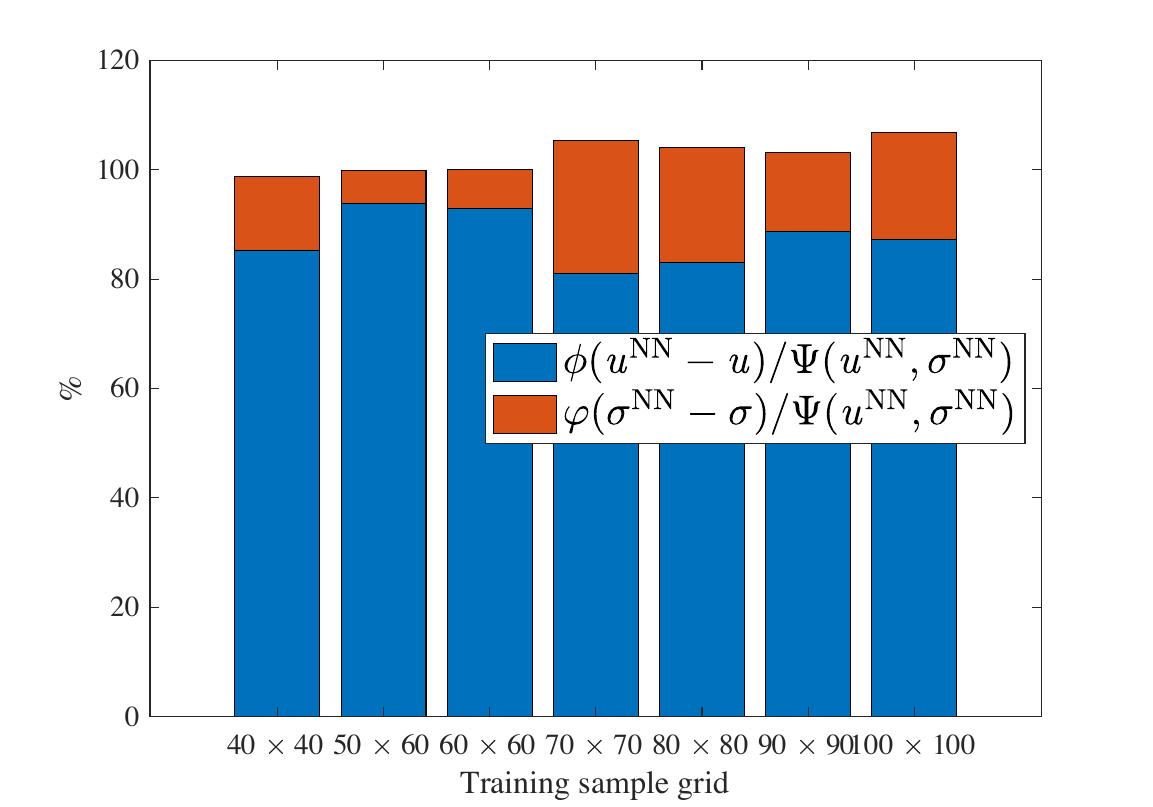}}
	\hfill
	\subfigure[Global generalization errors and CRE versus training sample grid]{\includegraphics[width=.48\textwidth]{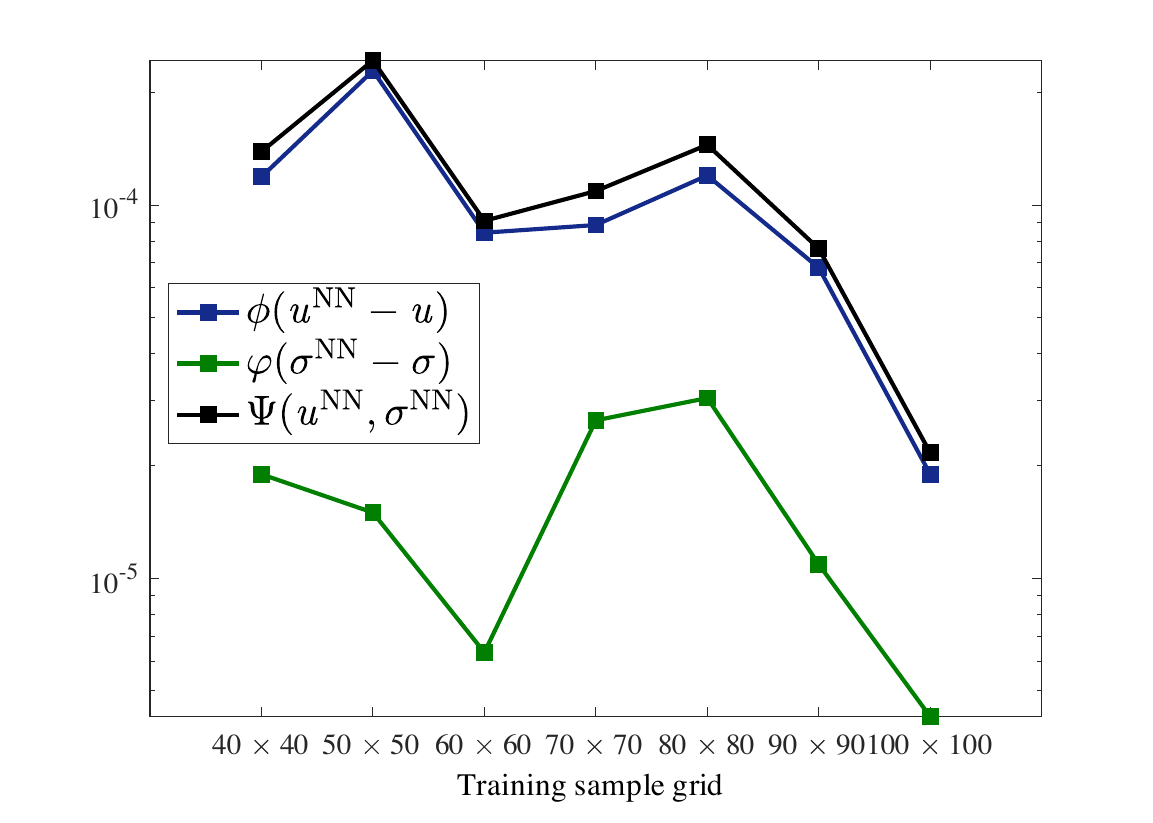}\label{rgnsample2}}

	\subfigure[$\phi/\Psi$ and $\varphi/\Psi$ versus number of neurons]{\includegraphics[width=.48\textwidth]{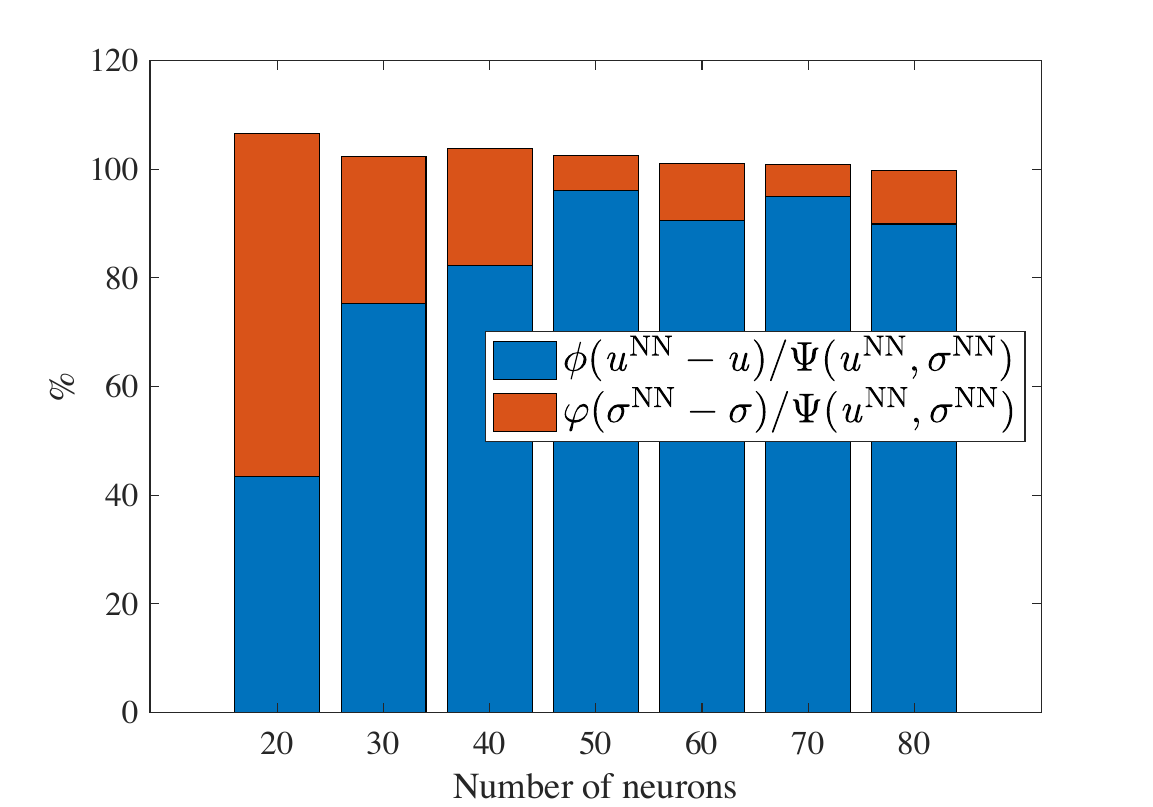}}
	\hfill
	\subfigure[Global generalization errors and CRE versus number of neurons]{\includegraphics[width=.48\textwidth]{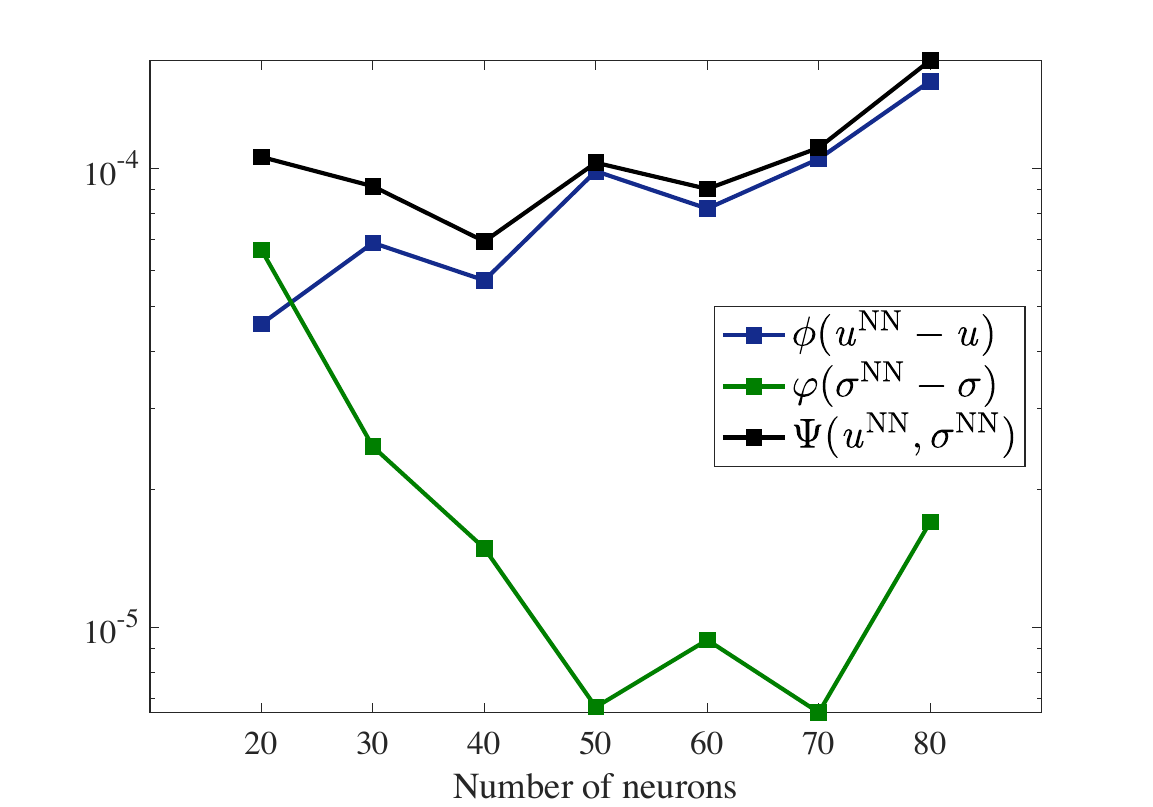}\label{rgnneuron2}}

	\subfigure[$\phi/\Psi$ and $\varphi/\Psi$ versus number of layers]{\includegraphics[width=.48\textwidth]{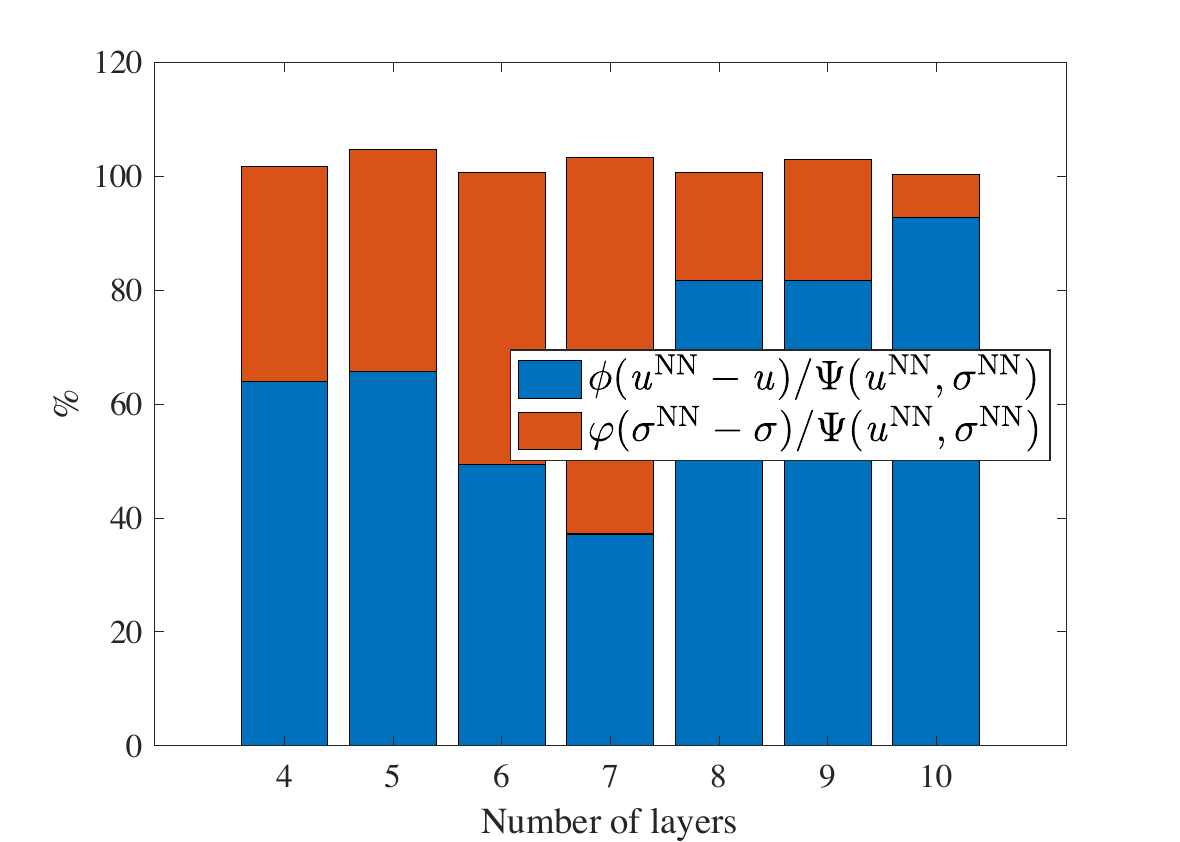}}
	\hfill
	\subfigure[Global generalization errors and CRE versus number of layers]{\includegraphics[width=.48\textwidth]{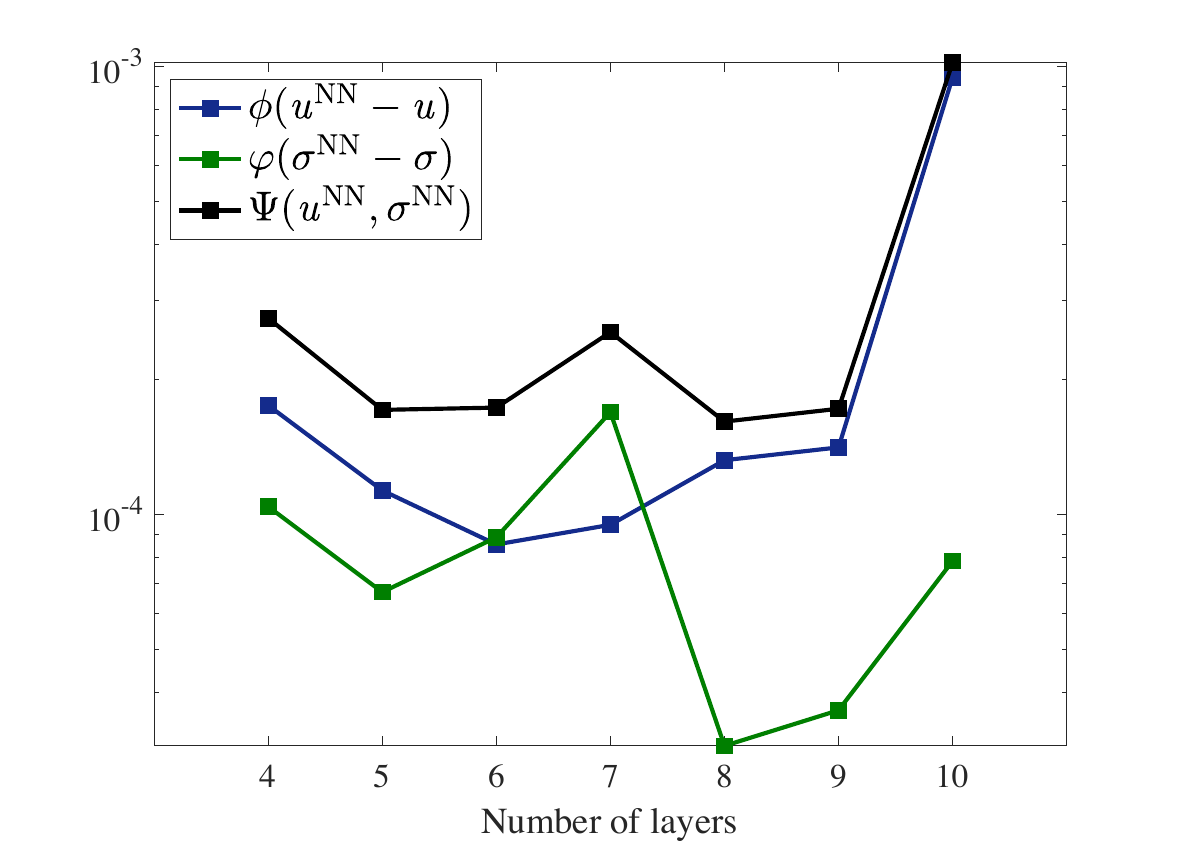}\label{rgnlayer2}}
	
	\caption{Comparison of the energy error in displacements $\phi(\vb*{u}^\texttt{NN}-\vb*{u})$, the energy error in stresses $\varphi(\vb*{\sigma}^\texttt{NN}-\vb*{\sigma})$, and the CRE $\Psi(\vb*{u}^\texttt{NN},\vb*{\sigma}^\texttt{NN})$ in the scenario (b) where PINN is used as a regression method.}
	\label{fig:rgnresults}
\end{figure}

\section{Conclusions}

In this note, we have briefly presented a preliminary work on \emph{a posteriori} error analysis for the PINN solutions of elasticity problems. An energy-based error bound, defined as a constitutive relation error, is employed to provide an upper bound of the global discretization errors in both the displacements and the stresses solved through a mix form of PINNs. The generalization performance of PINNs is thus assessed by the proposed error bound. In the numerical demonstration, a reference 2D linear elasticity problem is analyzed using PINNs as a forward solver and as a regression tool. Considering a wide range of network hyper parameters, i.e., number of  neurons and hidden layers, and number of sampling points, it has been verified that the proposed error bound guarantees an upper bound of the energy-based measures of generalization errors and offers good robustness. The constitutive relation error bound is also directly applicable for hyperelasticity problems, and can be used for goal-oriented error estimation in cooperation with an adjoint problem.

\clearpage

 \bibliographystyle{abbrv}
 \bibliography{refs}

\begin{thebibliography}{10}

\bibitem{ainsworth2011posteriori}
M.~Ainsworth and J.~T. Oden.
\newblock {\em A posteriori error estimation in finite element analysis},
  volume~37.
\newblock John Wiley \& Sons, 2011.

\bibitem{amodei2016speech}
D.~Amodei, S.~Ananthanarayanan, R.~Anubhai, J.~Bai, E.~Battenberg, C.~Case,
  J.~Casper, B.~Catanzaro, Q.~Cheng, G.~Chen, et~al.
\newblock Deep speech 2: End-to-end speech recognition in english and mandarin.
\newblock In {\em International conference on machine learning}, pages
  173--182, 2016.

\bibitem{babuvska1978posteriori}
I.~Babu{\v{s}}ka and W.~C. Rheinboldt.
\newblock A-posteriori error estimates for the finite element method.
\newblock {\em International Journal for Numerical Methods in Engineering},
  12(10):1597--1615, 1978.

\bibitem{bank1985some}
R.~E. Bank and A.~Weiser.
\newblock Some a posteriori error estimators for elliptic partial differential
  equations.
\newblock {\em Mathematics of Computation}, 44(170):283--301, 1985.

\bibitem{baydin2017automatic}
A.~G. Baydin, B.~A. Pearlmutter, A.~A. Radul, and J.~M. Siskind.
\newblock Automatic differentiation in machine learning: a survey.
\newblock {\em The Journal of Machine Learning Research}, 18(1):5595--5637,
  2017.

\bibitem{bekele2020biot}
Y.~W. Bekele.
\newblock Deep learning for one-dimensional consolidation.
\newblock {\em arXiv:2004.11689}, 2020.

\bibitem{deng2009imagenet}
J.~Deng, W.~Dong, R.~Socher, L.-J. Li, K.~Li, and L.~Fei-Fei.
\newblock Imagenet: A large-scale hierarchical image database.
\newblock In {\em 2009 IEEE conference on computer vision and pattern
  recognition}, pages 248--255. Ieee, 2009.

\bibitem{deuflhard1989concepts}
P.~Deuflhard, P.~Leinen, and H.~Yserentant.
\newblock Concepts of an adaptive hierarchical finite element code.
\newblock {\em IMPACT of Computing in Science and Engineering}, 1(1):3--35,
  1989.

\bibitem{gallimard2009constitutive}
L.~Gallimard.
\newblock A constitutive relation error estimator based on traction-free
  recovery of the equilibrated stress.
\newblock {\em International Journal for Numerical Methods in Engineering},
  78(4):460--482, 2009.

\bibitem{goodfellow2016deep}
I.~Goodfellow, Y.~Bengio, A.~Courville, and Y.~Bengio.
\newblock {\em Deep learning}, volume~1.
\newblock MIT press Cambridge, 2016.

\bibitem{graves2013speech}
A.~Graves, A.-r. Mohamed, and G.~Hinton.
\newblock Speech recognition with deep recurrent neural networks.
\newblock In {\em 2013 IEEE international conference on acoustics, speech and
  signal processing}, pages 6645--6649. IEEE, 2013.

\bibitem{grigorescu2020driving}
S.~Grigorescu, B.~Trasnea, T.~Cocias, and G.~Macesanu.
\newblock A survey of deep learning techniques for autonomous driving.
\newblock {\em Journal of Field Robotics}, 37(3):362--386, 2020.

\bibitem{guo2016legendre}
M.~Guo, W.~Han, and H.~Zhong.
\newblock Legendre-fenchel duality and a generalized constitutive relation
  error.
\newblock {\em arXiv:1611.05589}, 2016.

\bibitem{ha2016ecommerce}
J.-W. Ha, H.~Pyo, and J.~Kim.
\newblock Large-scale item categorization in e-commerce using multiple
  recurrent neural networks.
\newblock In {\em Proceedings of the 22nd ACM SIGKDD International Conference
  on Knowledge Discovery and Data Mining}, pages 107--115, 2016.

\bibitem{haghighat2020sciann}
E.~Haghighat and R.~Juanes.
\newblock Sciann: A keras wrapper for scientific computations and
  physics-informed deep learning using artificial neural networks.
\newblock {\em arXiv:2005.08803}, 2020.

\bibitem{haghighat2020solid}
E.~Haghighat, M.~Raissi, A.~Moure, H.~Gomez, and R.~Juanes.
\newblock A deep learning framework for solution and discovery in solid
  mechanics.
\newblock {\em arXiv:2003.02751}, 2020.

\bibitem{haghighat2020deep}
E.~Haghighat, M.~Raissi, A.~Moure, H.~Gomez, and R.~Juanes.
\newblock A deep learning framework for solution and discovery in solid
  mechanics.
\newblock {\em arXiv:2003.02751}, 2020.

\bibitem{jin2020fluid}
X.~Jin, S.~Cai, H.~Li, and G.~E. Karniadakis.
\newblock Nsfnets (navier-stokes flow nets): Physics-informed neural networks
  for the incompressible navier-stokes equations.
\newblock {\em arXiv:2003.06496}, 2020.

\bibitem{kadeethum2020biot}
T.~Kadeethum, T.~M. J{\o}rgensen, and H.~M. Nick.
\newblock Physics-informed neural networks for solving nonlinear diffusivity
  and biot’s equations.
\newblock {\em PloS one}, 15(5):e0232683, 2020.

\bibitem{kharazmi2019variational}
E.~Kharazmi, Z.~Zhang, and G.~E. Karniadakis.
\newblock Variational physics-informed neural networks for solving partial
  differential equations.
\newblock {\em arXiv:1912.00873}, 2019.

\bibitem{kharazmi2020varitionalhp}
E.~Kharazmi, Z.~Zhang, and G.~E. Karniadakis.
\newblock hp-vpinns: Variational physics-informed neural networks with domain
  decomposition.
\newblock {\em arXiv:2003.05385}, 2020.

\bibitem{krizhevsky2012imagenet}
A.~Krizhevsky, I.~Sutskever, and G.~E. Hinton.
\newblock Imagenet classification with deep convolutional neural networks.
\newblock In {\em Advances in neural information processing systems}, pages
  1097--1105, 2012.

\bibitem{ladeveze2016constitutive}
P.~Ladev{\`e}ze and L.~Chamoin.
\newblock The constitutive relation error method: A general verification tool.
\newblock In {\em Verifying Calculations-Forty Years On}, pages 59--94.
  Springer, 2016.

\bibitem{ladeveze1983error}
P.~Ladev\`{e}ze and D.~Leguillon.
\newblock Error estimate procedure in the finite element method and
  applications.
\newblock {\em SIAM Journal on Numerical Analysis}, 20(3):485--509, 1983.

\bibitem{ladeveze2005mastering}
P.~Ladev{\`e}ze and J.~P. Pelle.
\newblock {\em Mastering calculations in linear and nonlinear mechanics}.
\newblock Springer, 2005.

\bibitem{lecun2015deep}
Y.~LeCun, Y.~Bengio, and G.~Hinton.
\newblock Deep learning.
\newblock {\em nature}, 521(7553):436--444, 2015.

\bibitem{mao2020fluid}
Z.~Mao, A.~D. Jagtap, and G.~E. Karniadakis.
\newblock Physics-informed neural networks for high-speed flows.
\newblock {\em Computer Methods in Applied Mechanics and Engineering},
  360:112789, 2020.

\bibitem{pang2019fractionals}
G.~Pang, L.~Lu, and G.~E. Karniadakis.
\newblock fpinns: Fractional physics-informed neural networks.
\newblock {\em SIAM Journal on Scientific Computing}, 41(4):A2603--A2626, 2019.

\bibitem{pled2011techniques}
F.~Pled, L.~Chamoin, and P.~Ladev\`{e}ze.
\newblock On the techniques for constructing admissible stress fields in model
  verification: Performances on engineering examples.
\newblock {\em International Journal for Numerical Methods in Engineering},
  88(5):409--441, 2011.

\bibitem{raissi2018hidden}
M.~Raissi and G.~E. Karniadakis.
\newblock Hidden physics models: Machine learning of nonlinear partial
  differential equations.
\newblock {\em Journal of Computational Physics}, 357:125--141, 2018.

\bibitem{raissi2019pinn}
M.~Raissi, P.~Perdikaris, and G.~E. Karniadakis.
\newblock Physics-informed neural networks: A deep learning framework for
  solving forward and inverse problems involving nonlinear partial differential
  equations.
\newblock {\em Journal of Computational Physics}, 378:686--707, 2019.

\bibitem{raissi2020pinn}
M.~Raissi, A.~Yazdani, and G.~E. Karniadakis.
\newblock Hidden fluid mechanics: Learning velocity and pressure fields from
  flow visualizations.
\newblock {\em Science}, 367(6481):1026--1030, 2020.

\bibitem{rao2020solid}
C.~Rao, H.~Sun, and Y.~Liu.
\newblock Physics informed deep learning for computational elastodynamics
  without labeled data.
\newblock {\em arXiv:2006.08472}, 2020.

\bibitem{sallab2017driving}
A.~E. Sallab, M.~Abdou, E.~Perot, and S.~Yogamani.
\newblock Deep reinforcement learning framework for autonomous driving.
\newblock {\em Electronic Imaging}, 2017(19):70--76, 2017.

\bibitem{shankar2017ecommerce}
D.~Shankar, S.~Narumanchi, H.~Ananya, P.~Kompalli, and K.~Chaudhury.
\newblock Deep learning based large scale visual recommendation and search for
  e-commerce.
\newblock {\em arXiv:1703.02344}, 2017.

\bibitem{shin2020convergence}
Y.~Shin, J.~Darbon, and G.~E. Karniadakis.
\newblock On the convergence and generalization of physics informed neural
  networks.
\newblock {\em arXiv:2004.01806}, 2020.

\bibitem{yang2019fluid}
X.~Yang, S.~Zafar, J.-X. Wang, and H.~Xiao.
\newblock Predictive large-eddy-simulation wall modeling via physics-informed
  neural networks.
\newblock {\em Physical Review Fluids}, 4(3):034602, 2019.

\bibitem{yang2018generative}
Y.~Yang and P.~Perdikaris.
\newblock Physics-informed deep generative models.
\newblock {\em arXiv:1812.03511}, 2018.

\bibitem{zhao2020prediction}
X.~Zhao, K.~Shirvan, R.~K. Salko, and F.~Guo.
\newblock On the prediction of critical heat flux using a physics-informed
  machine learning-aided framework.
\newblock {\em Applied Thermal Engineering}, 164:114540, 2020.

\bibitem{zienkiewicz1987simple}
O.~C. Zienkiewicz and J.~Z. Zhu.
\newblock A simple error estimator and adaptive procedure for practical
  engineerng analysis.
\newblock {\em International Journal for Numerical Methods in Engineering},
  24(2):337--357, 1987.

\end{thebibliography}

\end{document}